\documentclass[preprint,review,12pt]{elsarticle}
\usepackage{amsmath}
\usepackage{amsfonts}
\usepackage{geometry}   
\usepackage{bbm}
\usepackage[parfill]{parskip}    
\usepackage{graphicx}
\usepackage{amssymb}
\usepackage{epstopdf}
\usepackage{psfrag}
\newcommand{\figref}[1]{{Figure~\ref{#1}}}
\newcommand {\be}{\begin{equation}}
\newcommand {\bes}{\begin{displaymath}}
\newcommand {\es}{\end{displaymath}}
\newcommand {\e}{\end{equation}}

\newcommand {\bea}{\begin{eqnarray}}
\newcommand {\ea}{\end{eqnarray}}

\newtheorem{theorem}{Theorem}[section]
\newtheorem{Assumption}{Assumption}[section]

\newtheorem{remark}{Remark}[section]

\newenvironment{proof}[1][Proof]{\textbf{#1.} }{\hspace{\stretch{1}}\rule{0.5em}{0.5em}}

\usepackage{xspace}

\usepackage{subfigure}

\newcommand{\thmref}[1]{{Theorem~\ref{#1}}}

\newcommand{\secref}[1]{{Section~\ref{#1}}}
\newcommand{\assref}[1]{{Assumption~\ref{#1}}}

\begin{document}
\begin{frontmatter}

\title{Stability of the  semi-tamed  and  tamed Euler schemes  for stochastic differential
 equations with jumps under non-global Lipschitz condition}

\author[at,atb]{Antoine Tambue}
\cortext[cor1]{Corresponding author}
\ead{antonio@aims.ac.za}
\address[at]{The African Institute for Mathematical Sciences(AIMS) of South Africa and Stellenbosh University,
6-8 Melrose Road, Muizenberg 7945, South Africa}
\address[atb]{Center for Research in Computational and Applied Mechanics (CERECAM), and Department of Mathematics and Applied Mathematics, University of Cape Town, 7701 Rondebosch, South Africa.}

\author[jd,jdm]{Jean Daniel Mukam}
\ead{jean.d.mukam@aims-senegal.org}
\address[jd]{African Institute for Mathematical Sciences(AIMS) of Senegal, KM 2, ROUTE DE JOAL, B.P. 1418. MBOUR, SENEGAL }
\address[jdm]{Technische Universit\"{a}t Chemnitz, 09126 Chemnitz, Germany}

\begin{abstract}
 Under non-global Lipschitz condition, Euler Explicit method fails to converge strongly to the exact solution, 
 while Euler implicit  method converges but requires much computational efforts. Tamed scheme was  first introduced in \cite{lamport194} to overcome this failure of  the standard explicit method.
 This  technique is extended to  SDEs driven by Poisson jump  in  \cite{antojeanst} where several schemes were analyzed.
 In this work,  we investigate their  nonlinear stability under non-global Lipschitz and  their linear stability. 
  Numerical simulations to sustain the theoretical results are also provided.
\end{abstract}

\end{frontmatter}
\section{Introduction}
\label{intro}
In this work, we study the stability of the jump-diffusion It\^{o}'s stochastic differential equations (SDEs) of the form 
\begin{eqnarray}
\label{model}
 dX(t)=  f(X(t^{-}))dt +g(X(t^{-}))dW(t)+h(X(t^{-}))dN(t), \quad  X(0)=X_0,\;\;  t  \in [0, T],\,\, T>0.
\end{eqnarray}
Here $W_t$ is a $m$-dimensional Brownian motion, $f :\mathbb{R}^d\longrightarrow\mathbb{R}^d$, $d \in \mathbb{N}$ 
satisfies the  one-sided Lipschitz condition and the polynomial growth condition,
the  functions $g : \mathbb{R}^d \longrightarrow\mathbb{R}^{d\times m}$ and $h :\mathbb{R}^d \longrightarrow\mathbb{R}^d$ satisfy
the globally Lipschitz condition, and $N_t$ is a one dimensional poisson process with parameter $\lambda$.
The one-sided Lipschitz function $f$ can be decomposed as $f=u+v$, where the function $u : \mathbb{R}^d\longrightarrow\mathbb{R}^d$ 
is the global Lipschitz continuous part and $v : \mathbb{R}^d\longrightarrow\mathbb{R}^d$ is the non-global Lipschitz continuous part. 
Using this decomposition, we can rewrite the jump-diffusion SDEs \eqref{model} in the following equivalent form
\begin{eqnarray}
\label{model1}
dX(t)=  \left(u(X(t^{-})+ v(X(t^{-}))\right)dt +g(X(t^{-}))dW(t)+h(X(t^{-}))dN(t).
\end{eqnarray}
Equations of type \eqref{model} arise in a range of scientific, engineering and financial applications (see \cite{appf2,appf,Fima} and references therein).
The standard explicit methods for approximating SDEs of  type \eqref{model} is  the Euler-Maruyama method  and implicit schemes \cite{Desmond2,platen}. 
Their numerical analysis   have been   studied in \cite{Desmond2,Wang,desmond,platen} with implicit and explicit schemes. 
Recently it has been proved (see \cite{lamport94}) that the Euler-Maruyama method often fails to converge strongly 
to the exact solution of nonlinear SDEs of the form  \eqref{model} without jump  term  when at least one of the  functions $f$ and $g$
 grows superlinearly. To  overcome  this drawback of the Euler-Maruyama method, numerical approximation 
which computational cost is close to that of the Euler-Maruyama method and which converge strongly even in the case  the function $f$  is superlinearly growing was first introduced in
\cite{lamport194}. In our accompanied paper \cite{antojeanst}, the work in \cite{lamport194} has been extended to SDEs of type \eqref{model}
and  the strong convergence of the following numerical schemes has been investigated
\begin{equation}
 X_{n+1}=X_{n}+\dfrac{\Delta t f(X_{n})}{1+ \Delta t\Vert f(X_{n}) \Vert }+g(X_{n}) \Delta W_n +h(X_{n})\Delta N_n,
 \label{ncts}
\end{equation}
and 
\begin{eqnarray}
 Y_{n+1}=Y_{n}+u(Y_n)\Delta t+\dfrac{\Delta t v(Y_{n})}{1+ \Delta t \Vert v(Y_{n}) \Vert }+g(Y_{n}) \Delta W_n +h(Y_{n})\Delta N_n.
 \label{sts}
\end{eqnarray}
where $\Delta t=T/M$ is the  time step-size, $M\in\mathbb{N}$ is the number of time subdivisions, $\Delta W_n = W(t_{n+1})- W(t_{n})$ and $\Delta N_n = N(t_{n+1})- N(t_{n})$.
The scheme \eqref{ncts} is called  the  non compensated tamed scheme (NCTS), while scheme \eqref{sts} is  called the semi-tamed scheme.

Strong and weak convergences are not the only features of numerical techniques. Stability for SDEs is also a good feature as 
the information about  step size for which  does a particular numerical method replicate the stability properties of the exact solution is valuable. The linear stability is an extension of  the deterministic A-stability
while exponential stability can guarantee that errors introduced in one time step will decay exponentially in future time steps,  exponential
stability  also implies asymptotic stability \cite{Huang}.  By the Chebyshev inequality and the Borel--Cantelli lemma, it is well known that exponential mean
square stability implies almost sure stability \cite{Huang}.  The stability of  classical implicit and explicit methods for  \eqref{model} are well understood \cite{Desmond2,Huang,Wang,xia}. 
Although the strong convergence of the NCTS  and STS schemes given respectively by  \eqref{ncts} and \eqref{sts} have been provided in \cite{antojeanst}, 
a rigorous stability properties have not yet investigated to the best of our knowledge. The  goal of this paper is to study  the linear stability and  the exponential stability of \eqref{ncts} and \eqref{sts} 
for SDEs \eqref{model} driven by both Brownian motion and  Poisson jump.
Our study will also provide the rigorous study of  linear stabilities of schemes   \eqref{ncts} and \eqref{sts} for SDEs  without jump,
which have not yet studied  to the best of our knowledge.

The paper is organised as follows. The linear mean-square stability  and  the exponential mean-square stability of the tamed and semi-tamed schemes are investigated  
respectively in \secref{linearsta} and \secref{nlinearsta}.  \secref{simulations} presents numerical simulations to sustain the theoretical results. We also
compare  the stability behaviors of tamed  and semi-tamed schemes  with  those of backward Euler and split-step backward Euler, this comparison shows  
the good behavior of  the semi-tamed scheme and therefore confirms the previous study in \cite{lamp} for SDEs without jump.

\section{Linear mean-square stability} 
\label{linearsta}
   Throughout this work,  $(\Omega, \mathcal{F}, \mathbb{P})$ denotes a complete probability space with a filtration $(\mathcal{F}_t)_{t\geq 0}$. 
 For all $x, y\in\mathbb{R}^d$, we denote by $ \langle x, y \rangle= x_1y_1+x_2y_2+\cdots+x_dy_d$, $\|x\|= \langle x, x \rangle^{1/2}$ and $ a \wedge b =\min (a,b)$, for all $a, b\in\mathbb{R}$.

 The goal of  this section is to find the time  
 step-size  limit for which the tamed Euler scheme and the semi-tamed Euler scheme are stable
 in  the linear mean-square sense. 
 For the scalar linear test problem, the concept of A-stability of a numerical method may be interpreted as ''problem
stable $\Rightarrow$ method stable for all $\Delta t$''.
We consider the following linear test equation with real and scalar coefficients.
\begin{eqnarray}
 dX(t)=  aX(t^{-})dt +bX(t^{-})dW(t)+cX(t^{-})dN(t),  \hspace{0.5cm}
 X(0)=X_0,
 \label{ch5linearequation}
\end{eqnarray}
where $X_0$ satisfied $\mathbb{E}\|X_0\|^2<\infty$.
It is proved in \cite{Desmond2} that the exact solution of \eqref{ch5linearequation} is mean-square stable if and only if 
  \begin{eqnarray}
  \label{lstable}
  \underset{t \rightarrow \infty}{\lim} \mathbb{E}(X(t)^2)=0  \Leftrightarrow l :=2a+b^2+\lambda c(2+c)<0.
  \end{eqnarray}
  Using the discrete form  of \eqref{ch5linearequation} ,  the numerical schemes \eqref{sts} and \eqref{ncts} will be therefore mean-square stable if  $l<0$ and 
   \begin{eqnarray}
   \underset{n \rightarrow \infty}{\lim} \mathbb{E} (Y_n^2)= \underset{n \rightarrow \infty}{\lim} \mathbb{E}(X_n^2)= 0.
   \end{eqnarray}
The following result  provides the  time step-size limit for which the semi-tamed scheme (STS) \eqref{sts} is  is mean-square stable.
 \begin{theorem}
 Assume  that $l<0$, then $a+\lambda c<0$ and  the semi-tamed  scheme \eqref{sts} is mean-square stable if and only if
 \begin{eqnarray*}
 \Delta t<\dfrac{-l}{(a+\lambda c)^2}.
 \end{eqnarray*}
 \end{theorem}
 \begin{proof}
 Applying the  semi-tamed Euler scheme to \eqref{ch5linearequation} leads to
 \begin{eqnarray}
 Y_{n+1}=Y_n+aY_n\Delta t+\lambda cY_n\Delta t+bY_n\Delta W_n+cY_n\overline{N}_n.
 \label{ch5compen1}
 \end{eqnarray}
 Squaring both sides of \eqref{ch5compen1} leads to
 \begin{eqnarray}
 Y_{n+1}^2&=&Y_n^2+(a+\lambda c)^2\Delta t^2Y_n^2+b^2Y_n^2\Delta W_n^2+c^2Y_n^2\Delta\overline{N}_n^2+2(a+\lambda c) \Delta t Y_n^2+2bY_n^2\Delta W_n\nonumber\\
 &+&2cY_n^2\Delta\overline{N}_n+2b(a+\lambda c)\Delta t\Delta W_n Y_n^2+2c(a+\lambda c)Y_n^2\Delta t\Delta\overline{N}_n+2bcY_n^2\Delta W_n\Delta\overline{N}_n.
 \label{ch5compen2}
 \end{eqnarray}
 Taking expectation in both sides of \eqref{ch5compen2} and using the relations  $\mathbb{E}(\Delta W_n^2)=\Delta t$, $\mathbb{E}(\Delta\overline{N}_n^2)=\lambda \Delta t$ and  $\mathbb{E}(\Delta W_n)=\mathbb{E}(\Delta\overline{N}_n)=0$ 
 with the fact  that $\Delta W_n$ and $ \Delta\overline{N}_n$ are independents leads to
 \begin{eqnarray*}
 \mathbb{E}|Y_{n+1}|^2=(1+(a+\lambda c)^2\Delta t^2+(b^2+\lambda c^2+2a+2\lambda c)\Delta t)\mathbb{E}|Y_n|^2.
 \end{eqnarray*}
 So, the semi-tamed scheme is stable if and only if
 \begin{eqnarray*}
 1+(a+\lambda c)^2\Delta t^2+(b^2+\lambda c^2+2a+2\lambda c)\Delta t<1.
 \end{eqnarray*}
 That is 
 $\Delta t<\dfrac{-l}{(a+\lambda c)^2}$.
 \end{proof}
 
 The following result  provide the  time step-size limit for which the non compensated  tamed scheme (NCTS) \eqref{ncts} is stable.
 \begin{theorem}
 \label{thmncts}
 Assume  that $l<0$,  the tamed Euler scheme \eqref{ncts} is mean-square stable  if one   of  the  following conditions is satisfied 
 
$\bullet$ $a(1+\lambda c\Delta t)\leq 0$, $2a-l>0$ and $\Delta t<\dfrac{2a-l}{a^2+\lambda^2c^2}$.\\
$\bullet$  $a(1+\lambda c\Delta t)> 0$ and $\Delta t<\dfrac{-l}{(a+\lambda c)^2}$.
\end{theorem}
\begin{proof}
Applying the tamed Euler scheme \eqref{ncts} to equation \eqref{ch5linearequation} leads to 
\begin{eqnarray}
X_{n+1}=X_n+\dfrac{a X_n\Delta t}{1+\Delta t|aX_n|}+bX_n\Delta W_n+cX_n\Delta N_n.
\label{ch5eq1}
\end{eqnarray}
By squaring both sides of \eqref{ch5eq1} leads to 
\begin{eqnarray*}
X_{n+1}^2&=&X_n^2+\dfrac{a^2X^2_n\Delta t^2}{(1+\Delta t|a X_n|)^2}+b^2 X_n^2\Delta W_n^2+c^2 X_n^2\Delta N_n^2+\dfrac{2a X_n^2\Delta t}{1+\Delta t|a X_n|}+2b X_n^2\Delta W_n\\
&+&2cX_n^2\Delta N_n+\dfrac{2abX_n^2}{1+\Delta t|aX_n|}\Delta W_n+\dfrac{2acX_n^2\Delta t}{1+\Delta t|a X_n|}\Delta N_n+2bcY_n^2\Delta W_n\Delta N_n.
\end{eqnarray*}
Using the inequality $ \dfrac{a^2\Delta t^2}{1+\Delta t|a X_n|}<a^2\Delta t^2$, the previous equality becomes 
\begin{eqnarray*}
X_{n+1}^2&\leq &X_n^2+a^2X^2_n\Delta t^2+b^2 X_n^2\Delta W_n^2+c^2 X_n^2\Delta N_n^2+\dfrac{2a X_n^2\Delta t}{1+\Delta t|a X_n|}+2bX_n^2\Delta W_n\\
&+&2c X_n^2\Delta N_n+\dfrac{2ab X_n^2}{1+\Delta t|aX_n|}\Delta W_n+\dfrac{2acX_n^2\Delta t}{1+\Delta t|aX_n|}\Delta N_n+2bcX_n^2\Delta W_n\Delta N_n.
\end{eqnarray*}
Taking expectation in both sides of the previous equality and using independence and  the fact that $\mathbb{E}(\Delta W_n)=0$, $\mathbb{E}(\Delta W_n^2)=\Delta t$, $\mathbb{E}(\Delta N_n)=\lambda\Delta t$, $\mathbb{E}(\Delta N_n^2)=\lambda \Delta t+\lambda^2\Delta t^2$ leads to :
\begin{eqnarray}
\mathbb{E}|X_{n+1}|^2&\leq& \left[1+a^2\Delta t^2+b^2\Delta t+\lambda^2c^2\Delta t^2+(2+ c)\lambda c\Delta t\right]\mathbb{E}|X_n|^2\nonumber\\
&+&\mathbb{E}\left(\dfrac{2aX^2_n\Delta t(1+\lambda c\Delta t)}{1+\Delta t|aX_n|}\right).
\label{ch5eq2}
\end{eqnarray}

$\bullet$  If $a(1+\lambda c\Delta t)\leq 0$, it follows from \eqref{ch5eq2} that
\begin{eqnarray*}
\mathbb{E}|X_{n+1}|^2\leq \{1+(a^2+\lambda^2c^2)\Delta t^2+[b^2+\lambda c(2+c)]\Delta t\}\mathbb{E}|X_n|^2.
\end{eqnarray*}
Therefore, the numerical solution is stable if 
\begin{eqnarray*}
1+(a^2+\lambda^2c^2)\Delta t^2+[b^2+\lambda c(2+c)]\Delta t<1.
\end{eqnarray*}
That is  $\Delta t<\dfrac{2a-l}{a^2+\lambda^2c^2}$.

$\bullet$  If $a(1+\lambda c\Delta t)> 0$, using the fact that $\dfrac{2a X_n^2\Delta t(1+\lambda c\Delta t)}{1+\Delta t|aX_n|}< 2aX_n^2\Delta t(1+\lambda c\Delta t)$, inequality \eqref{ch5eq2} becomes
\begin{eqnarray}
\mathbb{E}|X_{n+1}|^2\leq\left[1+a^2\Delta t^2+b^2\Delta t+\lambda^2c^2\Delta t^2+2\lambda ac\Delta t^2+(2+ c)\lambda c\Delta t+2a\Delta t\right]\mathbb{E}|X_n|^2.
\label{ch5eq3}
\end{eqnarray}
Therefore, it follows from \eqref{ch5eq3} that the numerical solution is stable if
 $1+a^2\Delta t^2+b^2\Delta t+\lambda^2c^2\Delta t^2+2\lambda ac\Delta t^2+(2+ c)\lambda c\Delta t+2a\Delta t<1$. That is $\Delta t<\dfrac{-l}{(a+\lambda c)^2}$.
\end{proof}
 \begin{remark}
In \thmref{thmncts}, we can  easily check that if $l<0$, we have:
\begin{eqnarray}
\left\lbrace \begin{array}{l}
 a(1+\lambda c\Delta t)\leq0,\\
 2a-l>0 \\
 \Delta t<\dfrac{2a-l}{a^2+\lambda^2c^2}\\
 \end{array} \right. 
 \Leftrightarrow 
 \left\lbrace \begin{array}{l}
   a \in (l/2,0], c\geq 0,\\
   \Delta t <\dfrac{2a-l}{a^2+\lambda^2 c^2}\\
 \end{array} \right.
 \bigcup
 \left\lbrace \begin{array}{l}
   a \in (l/2,0), c<0,\\
   \Delta t <\dfrac{2a-l}{a^2+\lambda^2 c^2} \\
   \Delta t\leq \dfrac{-1}{ \lambda c} \\
 \end{array} \right.\\
 \bigcup
 \left\lbrace \begin{array}{l}
   a>0, c<0 \\
   \Delta t <\dfrac{2a-l}{a^2+\lambda^2 c^2}  \\
  \Delta t \geq  \dfrac{-1}{ \lambda c}
 \end{array} \right. 
\end{eqnarray}
\begin{eqnarray}
\left\lbrace \begin{array}{l}
 a(1+\lambda c\Delta t)>0,\\
 \Delta t<\dfrac{-l}{(a+\lambda c)^2}\\
 \end{array} \right. 
 \Leftrightarrow 
 \left\lbrace \begin{array}{l}
   a >0, c>0,\\
   \Delta t < \dfrac{-l}{(a+\lambda c)^2}\\
 \end{array} \right.
 \bigcup
 \left\lbrace \begin{array}{l}
   a >0, c<0,\\
   \Delta t <\dfrac{-l}{(a+\lambda c)^2} \wedge \dfrac{-1}{ \lambda c} \\
 \end{array} \right.\\
 \bigcup
 \left\lbrace \begin{array}{l}
   a<0, c<0 \\
   \Delta t <\dfrac{-l}{(a+\lambda c)^2}  \\
  \Delta t >  \dfrac{-1}{ \lambda c}
 \end{array} \right. 
\end{eqnarray}
 \end{remark}

\section{Nonlinear mean-square stability}
\label{nlinearsta}
In this section, we focus on the  exponential mean-square stability of
the approximation \eqref{sts}. 
We follow closely \cite{lamp,Desmond2} and assume that $f(0)=g(0)=h(0)=0$  and $\mathbb{E}\Vert X_0 \Vert^2<\infty$.
It is proved in \cite{Desmond2} that under the following conditions, 
\begin{eqnarray}
\label{l1}
\langle x-y, f(x)-f(y)\rangle&\leq& \mu\|x-y\|^2,\\
\label{l2}
\|g(x)-g(y)\|^2&\leq& \sigma \|x-y\|^2,\\
\label{l3}
\|h(x)-h(y)\|^2 &\leq & \gamma \|x-y\|^2,
\end{eqnarray}
 for all $x,y\in\mathbb{R}^d$, where $\mu$, $\sigma$ and $ \gamma$ are constants, the exact solution of SDE \eqref{model} is  nonlinear mean-square stable  
 if $\alpha :=2\mu+\sigma+\lambda  \sqrt{\gamma}(\sqrt{\gamma}+2)<0$.
  Indeed under the above assumptions,  we have \cite[Theorem 4]{Desmond2}
 \begin{eqnarray*}
\mathbb{E}\Vert X(t) \Vert^2 \leq \mathbb{E}\Vert X_0 \Vert^2 e^{\alpha t}.
 \end{eqnarray*}
 So, if $\alpha :=2\mu+\sigma+\lambda  \sqrt{\gamma}(\sqrt{\gamma}+2)<0$ we have $ \underset { t \rightarrow \infty}{\lim} \mathbb{E}\Vert X(t) \Vert^2=0$ and  the exact solution $X$ is exponentially mean-square stable.
 In the sequel of  this section, we will  use some weak assumptions, which of courses imply that the conditions \eqref{l1}-\eqref{l3} hold.
 
  In order to study the nonlinear stability of  the semi-tamed scheme (STS), we  make also the following assumptions
  \begin{Assumption}
\label{ch5assumption2}
 There exist some  positive constants $\rho$, $\beta$,$\overline{\beta}$, $K$, $C$, $\theta$ and $a>1$  such that
 \begin{eqnarray*}
 \langle x-y, u(x)-u(y)\rangle\leq -\rho\|x-y\|^2,\hspace{1cm} \|u(x)-u(y)\|\leq K\|x-y\|,\\
 \langle x-y, v(x)-v(y)\rangle\leq-\beta\|x-y\|^{a+1}, \hspace{1.5cm} \|v(x)\|\leq \overline{\beta}\|x\|^a,\\
 \|g(x)-g(y)\|\leq\theta\|x-y\|,\hspace{2cm} \|h(x)-h(y)\|\leq C\|x-y\|.
 \end{eqnarray*}
 \end{Assumption}
We denote by $\alpha_1 =-2\rho+\theta^2+\lambda  C(C+2)$ and  we will always assume that $\alpha_1<0$  to ensure  the stability of  the exact solution.
 The nonlinear stability of STS scheme is given in the following theorem.
  \begin{theorem}
 \label{nt1}
 Under Assumptions \ref{ch5assumption2} and the further hypothesis $2\beta-\overline{\beta}>0$, 
 for any stepsize $\Delta t<\dfrac{-\alpha_1}{(K+\lambda C)^2}\wedge\dfrac{2\beta}{[2(K+\lambda C)+\overline{\beta}]\overline{\beta}}\wedge\dfrac{2\beta-\overline{\beta}}{2(K+\lambda C)\overline{\beta}}$,
  there exists  a constant  $\gamma=\gamma (\Delta t) >0$ such that
 \begin{eqnarray*}
 \mathbb{E}\|X_n\|^2\leq \mathbb{E}\|X_0\|^2 e^{-\gamma\, t_n}, \,\, t_n=n\,\Delta t,\qquad \underset{ \Delta t \rightarrow 0}{\lim}\gamma(\Delta t)= -\alpha_1,
 \end{eqnarray*}  
 and  the  numerical solution \eqref{sts} is exponentiallly mean-square stable.
 \end{theorem}
 \begin{proof}
 The numerical solution \eqref{sts} is given by 
 \begin{eqnarray*}
 Y_{n+1}=Y_n+\Delta tu_{\lambda}(Y_n)+\dfrac{\Delta tv(Y_n)}{1+\Delta t \|v(Y_n)\|}+g(Y_n)\Delta W_n+h(Y_n)\Delta\overline{N}_n,
 \end{eqnarray*}
  where $u_{\lambda}=u+\lambda h$.\\
 Taking the inner product in both sides  of the previous equation leads to 
 \begin{eqnarray}
 \|Y_{n+1}\|^2&=&\|Y_n\|^2+\Delta t^2\|u_{\lambda}(Y_n)\|^2+\dfrac{\Delta t^2 \|v(Y_n)\|^2}{\left(1+\Delta t\|v(Y_n)\|\right)^2}+\|g(Y_n)\|^2 \|\Delta W_n\|^2+\|h(Y_n)\|^2|\Delta \overline{N}_n |^2\nonumber\\
 &+&2\Delta t\langle Y_n, u_{\lambda}(Y_n)\rangle+2\Delta t\left\langle Y_n, \dfrac{v(Y_n)}{1+\Delta t\|v(Y_n)\|}\right\rangle+2\langle Y_n, g(Y_n)\Delta W_n\rangle\nonumber\\
 &+&2\langle Y_n, h(Y_n)\Delta\overline{N}_n\rangle+2\Delta t^2\left\langle u_{\lambda}(Y_n), \dfrac{v(Y_n)}{1+\Delta t||v(Y_n)||}\right\rangle\nonumber\\
 &+&2\Delta t\langle u_{\lambda}(Y_n), g(Y_n)\Delta W_n\rangle+2\Delta t\langle u_{\lambda}(Y_n), h(Y_n)\Delta\overline{N}_n\rangle\nonumber\\
 &+&2\Delta t\left\langle\dfrac{v(Y_n)}{1+\Delta t\|v(Y_n)\|}, g(Y_n)\Delta W_n\right\rangle+2\Delta t\left\langle\dfrac{v(Y_n)}{1+\Delta t\|v(Y_n)\|}, h(Y_n)\Delta\overline{N}_n\right\rangle\nonumber\\
 &+&2\langle g(Y_n)\Delta W_n, h(Y_n)\Delta\overline{N}_n\rangle.
 \label{ch5meansemi1}
 \end{eqnarray}
 Using Assumptions \ref{ch5assumption2}, it follows that 
 \begin{eqnarray}
 2\Delta t\left\langle Y_n, \dfrac{v(Y_n)}{1+\Delta \|v(Y_n)\|}\right\rangle\leq\dfrac{-2\beta\Delta t\|Y_n\|^{a+1}}{1+\Delta t\|v(Y_n)\|}
 \label{ch5meansemi2}
 \end{eqnarray}
 \begin{eqnarray}
 2\Delta t^2\left<u_{\lambda}(Y_n), \dfrac{v(Y_n)}{1+\Delta t\|v(Y_n)\|}\right>&\leq &\dfrac{2\Delta t^2\|u_{\lambda}(Y_n)\|\|v(Y_n)\|}{1+\Delta t\|v(Y_n\|}\nonumber\\
 &\leq&\dfrac{2\Delta t^2(K+\lambda C)\overline{\beta}\|Y_n\|^{a+1}}{1+\Delta t\|v(Y_n)\|}.
 \label{ch5meansemi3}
 \end{eqnarray}
 Set $\Omega_n=\{\omega\in\Omega : \|Y_n\|>1\}$. 
 
$\bullet$ On $\Omega_n$ we have
 \begin{eqnarray}
 \dfrac{\Delta t^2 \|v(Y_n)\|^2}{\left(1+\Delta t\|v(Y_n)\|\right)^2}\leq\dfrac{\Delta t\|v(Y_n)\|}{1+\Delta t\|v(Y_n)\|}\leq\dfrac{\overline{\beta}\Delta t\|Y_n\|^{a+1}}{1+\Delta t\|v(Y_n)\|}.
 \label{ch5meansemi4}
 \end{eqnarray}
 Therefore using \eqref{ch5meansemi2}, \eqref{ch5meansemi3} and \eqref{ch5meansemi4} in \eqref{ch5meansemi1} yields
\begin{eqnarray}
 \|Y_{n+1}\|^2& \leq &\|Y_n\|^2+\Delta t^2 \|u_{\lambda}(Y_n)\|^2+ \|g(Y_n)\|^2\|\Delta W_n\|^2+\|h(Y_n)\|^2|\Delta\overline{N}_n|^2\nonumber\\
 &+&2\Delta t\langle Y_n, u_{\lambda}(Y_n)\rangle+2\langle Y_n, g(Y_n)\Delta W_n\rangle
 +2\langle Y_n, h(Y_n)\Delta\overline{N}_n\rangle \nonumber\\
 &+&2\Delta t\langle u_{\lambda}(Y_n), g(Y_n)\Delta W_n\rangle+2\Delta t\langle u_{\lambda}(Y_n), h(Y_n)\Delta\overline{N}_n\rangle\nonumber\\
 &+&2\Delta t\left\langle\dfrac{v(Y_n)}{1+\Delta t\|v(Y_n)\|}, g(Y_n)\Delta W_n\right\rangle+2\Delta t\left\langle\dfrac{v(Y_n)}{1+\Delta t\|v(Y_n)\|}, h(Y_n)\Delta\overline{N}_n\right\rangle\nonumber\\
 &+&2\langle g(Y_n)\Delta W_n, h(Y_n)\Delta\overline{N}_n\rangle+\dfrac{\left[-2\beta\Delta t+2(K+\lambda c)\overline{\beta}\Delta t^2+\overline{\beta}\Delta t\right]\|Y_n\|^{a+1}}{1+\Delta t\|v(Y_n)\|}.
 \label{ch5meansemi4a}
 \end{eqnarray} 
 For $\Delta t<\dfrac{2\beta-\overline{\beta}}{2(K+\lambda C)\overline{\beta}}$, which is equivalent  to $-2\beta\Delta t+2(K+\lambda C)\overline{\beta}\Delta t^2+\overline{\beta}\Delta t<0$,  \eqref{ch5meansemi4a} becomes
 \begin{eqnarray}
 \|Y_{n+1}\|^2&\leq & \|Y_n\|^2+\Delta t^2\|u_{\lambda}(Y_n)\|^2+2\Delta t\langle Y_n, u_{\lambda}(Y_n)\rangle+\|g(Y_n)\|^2\|\Delta W_n\|^2\nonumber\\
 &+&\|h(Y_n)\|^2 |\Delta\overline{N}_n|^2+2\langle Y_n,g(Y_n)\Delta W_n\rangle+2\langle Y_n, h(Y_n)\Delta\overline{N}_n\rangle\nonumber\\
 &+&2\Delta t\langle u_{\lambda}(Y_n), g(Y_n)\Delta W_n\rangle+2\Delta t\langle u_{\lambda}(Y_n), h(Y_n)\Delta\overline{N}_n\rangle\nonumber\\
 &+&2\Delta t\left\langle \dfrac{v(Y_n)}{1+\Delta t \|v(Y_n)\|}, g(Y_n)\Delta W_n\right\rangle+2\Delta t\left\langle\dfrac{v(Y_n)}{1+\Delta t\|v(Y_n)\|}, h(Y_n)\Delta\overline{N}_n\right\rangle\nonumber\\
 &+&2\langle g(Y_n)\Delta W_n, h(Y_n)\Delta\overline{N}_n\rangle.
 \label{ch5meansemi4b}
 \end{eqnarray}
 
 $\bullet$ On $\Omega_n^c$ we have 
 \begin{eqnarray}
 \dfrac{\Delta t^2 \|v(Y_n)\|^2}{\left(1+\Delta t \|v(Y_n)\|\right)^2}\leq\dfrac{\Delta t^2\|v(Y_n)\|^2}{1+\Delta t\|v(Y_n)\|}\leq\dfrac{\overline{\beta}^2\Delta t^2\|Y_n\|^{2a}}{1+\Delta t
 \|v(Y_n)\|}\leq\dfrac{\overline{\beta}^2\Delta t^2 \|Y_n\|^{a+1}}{1+\Delta t\|v(Y_n)\|}.
 \label{ch5meansemi5}
 \end{eqnarray}
 Therefore, using \eqref{ch5meansemi2}, \eqref{ch5meansemi3} and \eqref{ch5meansemi5} in \eqref{ch5meansemi1} yields
  \begin{eqnarray}
 \|Y_{n+1}\|^2&\leq &\|Y_n\|^2+\Delta t^2 \|u_{\lambda}(Y_n)\|^2+2\Delta t\langle Y_n, u_{\lambda}(Y_n)\rangle+\|g(Y_n)\|^2 \|\Delta W_n\|^2\nonumber\\
 &+&\|h(Y_n)\|^2 |\Delta\overline{N}_n|^2+2\langle Y_n,g(Y_n)\Delta W_n\rangle+2\langle Y_n, h(Y_n)\Delta\overline{N}_n\rangle\nonumber\\
 &+&2\Delta t\langle u_{\lambda}(Y_n), g(Y_n)\Delta W_n\rangle+2\Delta t\langle u_{\lambda}(Y_n), h(Y_n)\Delta\overline{N}_n\rangle\nonumber\\
 &+&2\Delta t\left\langle \dfrac{v(Y_n)}{1+\Delta t\|v(Y_n)\|}, g(Y_n)\Delta W_n\right\rangle+2\Delta t\left\langle\dfrac{v(Y_n)}{1+\Delta t\|v(Y_n)\|}, h(Y_n)\Delta\overline{N}_n\right\rangle\nonumber\\
 &+&2\langle g(Y_n)\Delta W_n, h(Y_n)\Delta\overline{N}_n\rangle+\dfrac{\left[-2\beta\Delta t+2(K+\lambda c)\overline{\beta}\Delta t^2+\overline{\beta}^2\Delta t^2\right]\|Y_n\|^{a+1}}{1+\Delta t  \|v(Y_n)\|}.
 \label{ch5meansemi5a}
 \end{eqnarray}
 For  $\Delta t<\dfrac{2\beta}{[2(K+\lambda C)+\overline{\beta}]\overline{\beta}}$, which is equivalent to $-2\beta\Delta t+2(K+\lambda C)\overline{\beta}\Delta t^2+\overline{\beta}^2\Delta t^2<0$,
  \eqref{ch5meansemi5a} becomes 
 \begin{eqnarray}
 \|Y_{n+1}\|^2&\leq &\|Y_n\|^2+\Delta t^2\|u_{\lambda}(Y_n)\|^2+2\Delta t\langle Y_n, u_{\lambda}(Y_n)\rangle+\|g(Y_n)\|^2\|\Delta W_n\|^2\nonumber\\
 &+&\|h(Y_n)\|^2 |\Delta\overline{N}_n|^2+2\langle Y_n,g(Y_n)\Delta W_n\rangle+2\langle Y_n, h(Y_n)\Delta\overline{N}_n\rangle\nonumber\\
 &+&2\Delta t\langle u_{\lambda}(Y_n), g(Y_n)\Delta W_n\rangle+2\Delta t\langle u_{\lambda}(Y_n), h(Y_n)\Delta\overline{N}_n\rangle\nonumber\\
 &+&2\Delta t\left\langle \dfrac{v(Y_n)}{1+\Delta t \|v(Y_n)\|}, g(Y_n)\Delta W_n\right\rangle+2\Delta t\left\langle\dfrac{v(Y_n)}{1+\Delta t\|v(Y_n)\|}, h(Y_n)\Delta\overline{N}_n\right\rangle\nonumber\\
 &+&2\langle g(Y_n)\Delta W_n, h(Y_n)\Delta\overline{N}_n\rangle.
 \label{ch5meansemi5b}
 \end{eqnarray}
  Finally, from the discussion above on $\Omega_n$ and $\Omega_n^c$, it follows that on $\Omega$,  if $\Delta t \leq \dfrac{2\beta}{[2(K+\lambda C)+\overline{\beta}]\overline{\beta}}\wedge\dfrac{2\beta-\overline{\beta}}{2(K+\lambda C)\overline{\beta}}$ then we have
 \begin{eqnarray}
 \|Y_{n+1}\|^2&\leq &\|Y_n\|^2+\Delta t^2 \|u_{\lambda}(Y_n)\|^2+2\Delta t\langle Y_n, u_{\lambda}(Y_n)\rangle+\|g(Y_n)\|^2 \|\Delta W_n\|^2\nonumber\\
 &+&\|h(Y_n)\|^2||\Delta\overline{N}_n\|^2+2\langle Y_n,g(Y_n)\Delta W_n\rangle+2\langle Y_n, h(Y_n)\Delta\overline{N}_n\rangle\nonumber\\
 &+&2\Delta t\langle u_{\lambda}(Y_n), g(Y_n)\Delta W_n\rangle+2\Delta t\langle u_{\lambda}(Y_n), h(Y_n)\Delta\overline{N}_n\rangle\nonumber\\
 &+&2\Delta t\left\langle \dfrac{v(Y_n)}{1+\Delta t\|v(Y_n)\|}, g(Y_n)\Delta W_n\right\rangle+2\Delta t\left\langle\dfrac{v(Y_n)}{1+\Delta t \|v(Y_n)\|}, h(Y_n)\Delta\overline{N}_n\right\rangle\nonumber\\
 &+&2\langle g(Y_n)\Delta W_n, h(Y_n)\Delta\overline{N}_n\rangle.
 \label{ch5meansemi6}
 \end{eqnarray}
 Taking the expectation in both sides of \eqref{ch5meansemi6} and using the martingale properties of $\Delta W_n$ and $\Delta\overline{N}_n$ leads to 
 \begin{eqnarray}
 \mathbb{E}\|Y_{n+1}\|^2&\leq&\mathbb{E}\|Y_n\|^2+\Delta t^2\mathbb{E}\|u_{\lambda}(Y_n)\|^2+2\Delta t\mathbb{E}\langle Y_n, u_{\lambda}(Y_n)\rangle+\Delta t\mathbb{E}\|g(Y_n)\|^2\nonumber\\
 &+&\lambda\Delta t\mathbb{E}\|h(Y_n)\|^2.
 \label{ch5meansemi7}
 \end{eqnarray}
 From Assumptions \ref{ch5assumption2}, we have 
 \begin{eqnarray*}
 \|u_{\lambda}(Y_n)\|^2\leq (K+\lambda C)^2\|Y_n\|^2 \hspace{0.5cm} \text{and} \hspace{0.5cm}\langle Y_n, u_{\lambda}(Y_n)\rangle\leq (-\rho+\lambda C)\|Y_n\|^2.
 \end{eqnarray*}
 So inequality \eqref{ch5meansemi7} gives
 \begin{eqnarray*}
 \mathbb{E}\|Y_{n+1}\|^2&\leq&\mathbb{E}\|Y_n\|^2+(K+\lambda C)^2\Delta t^2\mathbb{E}\|Y_n\|^2+2(-\rho+\lambda C)\Delta t\mathbb{E}\|Y_n\|^2+\theta^2\Delta t\mathbb{E}\|Y_n\|^2\\
 &+&\lambda C^2\Delta t\mathbb{E}\|Y_n\|^2\\
 &=&\left[1-2\rho\Delta t+(K+\lambda C)^2\Delta t^2+2\lambda C\Delta t+\theta^2\Delta t+\lambda C^2\Delta t\right]\mathbb{E}\|Y_n\|^2.
 \end{eqnarray*}
 Iterating the previous inequality leads to  
 \begin{eqnarray*}
 \mathbb{E}\|Y_n\|^2\leq\left[1-2\rho\Delta t+(K+\lambda C)^2\Delta t^2+2\lambda C\Delta t+\theta^2\Delta t+\lambda C^2\Delta t\right]^n\mathbb{E}\| X_0\|^2
 \end{eqnarray*}
 for $\Delta t \leq \dfrac{2\beta}{[2(K+\lambda C)+\overline{\beta}]\overline{\beta}}\wedge\dfrac{2\beta-\overline{\beta}}{2(K+\lambda c)\overline{\beta}}$. 
 
The stability occurs  if and only if  $\underset{n \rightarrow \infty}{\lim} \mathbb{E}\Vert Y_n \Vert^2= 0$, so we should  also have 
 \begin{eqnarray*}
 1-2\rho\Delta t+(K+\lambda C)^2\Delta t^2+2\lambda C\Delta t+\theta^2\Delta t+\lambda C^2\Delta t<1.
 \end{eqnarray*}
 That is 
 \begin{eqnarray}
 \Delta t<\dfrac{-[-2\rho+\theta^2+\lambda C(2+C)]}{(K+\lambda C)^2} =\dfrac{-\alpha_1}{(K+\lambda C)^2},
 \end{eqnarray}
 and there exists  a constant $\gamma=\gamma (\Delta t) >0$ such that
 \begin{eqnarray*}
 \mathbb{E}\|X_n\|^2\leq\left[1-2\rho\Delta t+(K+\lambda C)^2\Delta t^2+2\lambda C\Delta t+\theta^2\Delta t+\lambda C^2\Delta t\right]^n\mathbb{E}\| X_0\|^2\leq \mathbb{E}\|X_0\|^2 e^{-\gamma t_n}, 
 \end{eqnarray*}
 By the Taylor expansion, as 
 \begin{eqnarray*}
 \lefteqn {\ln(1-2\rho \Delta t+(K+\lambda C)^2\Delta t^2+2\lambda C\Delta t+\theta^2\Delta t+\lambda C^2\Delta t)}&& \\
 &=& -2\rho\Delta t+(K+\lambda C)^2\Delta t^2+2\lambda C\Delta t+\theta^2\Delta t+\lambda C^2\Delta t\ +...... ,
 \end{eqnarray*}
 we obviously have  $\underset{ \Delta t \rightarrow 0}{\lim}\gamma(\Delta t)=- (-2 \rho +\theta^2+\lambda C(2+C))=-\alpha_1$.
 \end{proof}
 
  In order to  analyse the nonlinear mean-square stability of the tamed Euler scheme (NCTS),  we use the following assumption. 
 
 \begin{Assumption}
 \label{ch5assumption3}
 There exist some  positive constant $\beta$, $\overline{\beta}$,  $\theta$, $\mu$, $K$,  $\rho$, $C$  and $a>1$ such that :
 \begin{align}
 \langle x-y, f(x)-f(y)\rangle\leq &-\rho \|x-y\|^2-\beta \|x-y\|^{a+1},\nonumber\\
 \|f(x)\| \leq \overline{\beta}\| x \|^a+K\|x\|,\nonumber\\
 \|g(x)-g(y)\| \leq &\theta\|x-y\|,\qquad \,\,\|h(x)-h(y)\| \leq C\|x-y\|,\nonumber\\
  \langle x-y, h(x)-h(y)\rangle \leq &-\mu \|x-y\|^2.
  \label{assumparticular}
 \end{align}
 \end{Assumption}
 \begin{remark}
 Apart from \eqref{assumparticular}, \assref{ch5assumption3} is a consequence of \assref{ch5assumption2}.
 \end{remark}
 
  Using \assref{ch5assumption3},  we can easily check  that the exact solution of \eqref{model} is  exponentiallly mean-square stable if
  $ \alpha_{2}:=  -2\rho +\theta^2+\lambda C(2 +C)<0$.
 \begin{theorem}
  Under \assref{ch5assumption3}, if  
 $\beta-C\overline{\beta}>0$, $\overline{\beta}(1+2C)-2\beta<0$, and $\alpha_3 :=K+\theta^2+\lambda C(2K+C)-2\lambda\mu C<0$, for any stepsize
 \begin{eqnarray*}
\Delta t<\dfrac{-\alpha_3}{2K^2+\lambda^2C^2}\wedge\dfrac{\beta-C\overline{\beta}}{\overline{\beta}^2},
 \end{eqnarray*}
 there exists  a constant $\gamma=\gamma (\Delta t) >0$ such that
 \begin{eqnarray*}
 \mathbb{E}\|X_n\|^2\leq \mathbb{E}\|X_0\|^2 e^{-\gamma  t_n}, \,\, t_n=n\, \Delta t,  \underset{ \Delta t \rightarrow 0}{\lim}\gamma(\Delta t)=-\alpha_3.
 \end{eqnarray*}  
 and  the  numerical solution \eqref{ncts} is exponentiallly mean-square stable.
 \end{theorem}

\begin{proof}
 From equation \eqref{ncts}, we have 
 \begin{eqnarray}
 \|X_{n+1}\|^2&=&\|X_n\|^2+\dfrac{\Delta t^2 \|f(X_n)\|^2}{\left(1+\Delta t\|f(X_n)\|\right)^2}+\|g(X_n)\Delta W_n\|^2+\|h(X_n)\Delta N_n\|^2\nonumber\\
 &+&2\left\langle X_n,\dfrac{\Delta tf(X_n)}{1+\Delta t \|f(X_n)\|}\right\rangle+2\left\langle X_n+\dfrac{\Delta tf(X_n)}{1+\Delta t\|f(X_n)\|}, g(X_n)\Delta W_n\right\rangle\nonumber\\
 &+&2\left\langle X_n+\dfrac{\Delta tf(X_n)}{1+\Delta t\|f(X_n)\|}, h(X_n)\Delta N_n\right\rangle+2\langle g(X_n)\Delta W_n, h(X_n)\Delta N_n\rangle.
 \label{ch5meantamed1}
 \end{eqnarray}
 Using  \assref{ch5assumption3}, it follows that 
 \begin{eqnarray*}
 2\left\langle X_n, \dfrac{\Delta t f(X_n)}{1+\Delta t \|f(X_n)\|}\right\rangle &\leq &\dfrac{-2\Delta t\rho\|X_n\|^2}{1+\Delta t\|f(X_n)\|}-\dfrac{2\beta \Delta t \|X_n\|^{a+1}}{1+\Delta t \|f(X_n)\|}\\
 &\leq& -\dfrac{2\beta \Delta t\|X_n\|^{a+1}}{1+\Delta t\|f(X_n)\|}.
 \label{ch5meantamed2} \\
 \|g(X_n)\Delta W_n\|^2 &\leq& \theta^2 \|X_n\|^2\|\Delta W_n\|^2\\
 \|h(X_n)\Delta N_n \|^2 &\leq& C^2\|X_n\|^2|\Delta N_n|^2.\label{ch5meantamed2a}\\
 2\langle X_n, h(X_n)\Delta N_n\rangle &=& 2\langle X_n, h(X_n)\rangle\Delta N_n \leq -2\mu\|X_n\|^2|\Delta N_n|,\\
 \label{ch5meantamed3}\\
 2\left\langle\dfrac{\Delta tf(X_n)}{1+\Delta t \|f(X_n)\|}, h(X_n)\Delta N_n\right\rangle &\leq &\dfrac{2\Delta t\|f(X_n)\|\|h(X_n)\||\Delta N_n|}{1+\Delta t\|f(X_n)\|}\nonumber\\
 &\leq &\dfrac{2\Delta tC\overline{\beta}\|X_n\|^{a+1}}{1+\Delta t\|f(X_n)\|}|\Delta N_n|+2CK\|X_n\|^2|\Delta N_n|
 \end{eqnarray*}

 So from \assref{ch5assumption3}, we have
 \begin{eqnarray}
\label{ch5assmeantamed} \left\{
 \begin{array}{lllll}
 \left\langle X_n, \dfrac{\Delta tf(X_n)}{1+\Delta t \|f(X_n)\|}\right\rangle\leq -\dfrac{2\beta \Delta t\|X_n\|^{a+1}}{1+\Delta t\|f(X_n)\|}\\
 
 \|g(X_n)\Delta W_n\|^2 \leq \theta^2 \|X_n\|^2\|\Delta W_n\|^2\\
 
 \|h(X_n)\Delta N_n\|^2 \leq C^2\|X_n\|^2|\Delta N_n|^2\\
 
  2\langle X_n, h(X_n)\Delta N_n\rangle \leq -2\mu\|X_n\|^2|\Delta N_n|\\
  
   2\left\langle\dfrac{\Delta tf(X_n)}{1+\Delta t \|f(X_n)\|}, h(X_n)\Delta N_n\right\rangle \leq\dfrac{2\Delta tC\overline{\beta}\|X_n\|^{a+1}}{1+\Delta t\|f(X_n)\|}|\Delta N_n|+2CK \|X_n\|^2|\Delta N_n|

 \end{array}
 \right.
 \end{eqnarray}
 
 Let us define $\Omega_n :=\{w\in\Omega : \|X_n(\omega)\|>1\}$.
 
 $\bullet$  On $\Omega_n$,  using Assumption \ref{ch5assumption3} we have : 
 \begin{eqnarray}
 \dfrac{\Delta t^2\|f(X_n)\|^2}{\left(1+\Delta t \|f(X_n)\|\right)^2}&\leq& \dfrac{\Delta t\|f(X_n)\|}{1+\Delta t \|f(X_n)\|}
 \leq \dfrac{\Delta t\overline{\beta}\|X_n\|^a}{1+\Delta t\|f(X_n)\|}+K\Delta t\|X_n\|\nonumber\\
 &\leq& \dfrac{\Delta t\overline{\beta}\|X_n\|^{a+1}}{1+\Delta t\|f(X_n)\|}+K\Delta t \|X_n\|^2.
 \label{ch5meantamed5}
 \end{eqnarray} 
 Therefore using \eqref{ch5assmeantamed} and \eqref{ch5meantamed5} in \eqref{ch5meantamed1} yields 
 \begin{eqnarray}
 \|X_{n+1}\|^2&\leq&\|X_n\|^2+K\Delta t\|X_n\|^2+\theta^2\|X_n\|^2\|\Delta W_n\|^2+C^2 \|X_n\|^2|\Delta N_n|^2\nonumber\\
 &+&2\left\langle X_n+\dfrac{\Delta t f(X_n)}{1+\Delta t\|f(X_n)\|}, g(X_n)\Delta W_n\right\rangle-2\mu\|X_n\|^2|\Delta N_n|+2CK|\Delta N_n|\nonumber\\
 &+&2\langle g(X_n)\Delta W_n, h(X_n)\Delta N_n\rangle+\dfrac{\left[-2\beta\Delta t+\overline{\beta}\Delta t+2\overline{\beta}C\Delta t\right]\|X_n\|^{a+1}}{1+\Delta t\|f(X_n)\|}.
 \label{ch5meantamed6}
 \end{eqnarray}
 Since  $\overline{\beta}(1+2C)-2\beta<0$, \eqref{ch5meantamed6} becomes 
 \begin{eqnarray}
 \|X_{n+1}\|^2&\leq&\|X_n\|^2+K\Delta t\|X_n\|^2+\theta^2\|X_n\|^2 \|\Delta W_n\|^2+C^2\|X_n\|^2|\Delta N_n|^2\nonumber\\
 &+&2\left\langle X_n+\dfrac{\Delta tf(X_n)}{1+\Delta t\|f(X_n)\|}, g(X_n)\Delta W_n\right\rangle-2\mu\|X_n\|^2|\Delta N_n|+2CK|\Delta N_n|\nonumber\\
 &+&2\langle g(X_n)\Delta W_n, h(X_n)\Delta N_n\rangle.
 \label{ch5meantamed7}
 \end{eqnarray}
 $\bullet$ On $\Omega_n^c$, using Assumption \ref{ch5assumption3} and the inequality $(a+b)^2\leq 2a^2+2b^2$ we have 
 \begin{eqnarray}
 \dfrac{\Delta t^2 \|f(X_n)\|^2}{\left(1+\Delta t\|f(X_n)\|\right)^2}&\leq& \dfrac{\Delta t^2\|f(X_n)\|^2}{1+\Delta t\|f(X_n)\|}
 \leq \dfrac{2\Delta t^2\overline{\beta}^2 \|X_n\|^{2a}}{1+\Delta t\|f(X_n)\|}+2K^2\Delta t^2\|X_n\|^2\nonumber\\
 &\leq& \dfrac{2\Delta t^2\overline{\beta}^2\|X_n\|^{a+1}}{1+\Delta t\|f(X_n)\|}+2K^2\Delta t^2\|X_n\|^2.
 \label{ch5meantamed8}
 \end{eqnarray}
 Therefore, using \eqref{ch5assmeantamed} and \eqref{ch5meantamed8}, \eqref{ch5meantamed1} becomes
 \begin{eqnarray}
 \|X_{n+1}\|^2&\leq&\|X_n \|^2+2K^2\Delta t^2 \| X_n\|^2+\theta^2 \|X_n\|^2\|\Delta W_n\|^2+C^2\|X_n\|^2|\Delta N_n|^2\nonumber\\
 &+&2\left\langle X_n+\dfrac{\Delta tf(X_n)}{1+\Delta t\|f(X_n)\|}, g(X_n)\Delta W_n\right\rangle-2\mu\|X_n\|^2|\Delta N_n|+2CK|\Delta N_n|\nonumber\\
 &+&2\langle g(X_n)\Delta W_n, h(X_n)\Delta N_n\rangle+\dfrac{\left[2C\overline{\beta}\Delta t-2\beta\Delta t+2\overline{\beta}^2\Delta t^2\right] \|X_n\|^{a+1}}{1+\Delta t \|f(X_n)\|}.
 \label{ch5meantamed9}
 \end{eqnarray}
 For $\Delta t<\dfrac{\beta-C\overline{\beta}}{\overline{\beta}^2}$, which is equivalent to $2C\overline{\beta}\Delta t-2\beta\Delta t+2\overline{\beta}^2\Delta t^2<0$, 
 \eqref{ch5meantamed9} becomes 
 \begin{eqnarray}
 \|X_{n+1}\|^2&\leq&\|X_n\|^2+2K^2\Delta t^2\|X_n\|^2+\theta^2\|X_n\|^2\|\Delta W_n\|^2+C^2\|X_n\|^2|\Delta N_n|^2\nonumber\\
 &+&2\left\langle X_n+\dfrac{\Delta tf(X_n)}{1+\Delta t\|f(X_n)\|}, g(X_n)\Delta W_n\right\rangle-2\mu \|X_n\|^2|\Delta N_n|+2CK|\Delta N_n|\nonumber\\
 &+&2\langle g(X_n)\Delta W_n, h(X_n)\Delta N_n\rangle.
 \label{ch5meantamed10}
 \end{eqnarray}
 
 
From the above discussion on $\Omega_n$ and $\Omega_n^c$, it follows that on $\Omega$, if  $\Delta t<\dfrac{\beta-C\overline{\beta}}{\overline{\beta}^2}$ and $\overline{\beta}(1+2C)-2\beta<0$ then we have
 \begin{eqnarray}
 \|X_{n+1}\|^2&\leq&\|X_n\|^2+K\Delta t\|X_n\|^2+2K^2\Delta t^2\|X_n\|^2+\theta^2\|X_n\|^2\|\Delta W_n\|^2+C^2\|Y_n\|^2|\Delta N_n|^2\nonumber\\
 &+&2\left\langle X_n+\dfrac{\Delta tf(X_n)}{1+\Delta t\|f(X_n)\|}, g(X_n)\Delta W_n\right\rangle-2\mu\|X_n\|^2|\Delta N_n|+2CK|\Delta N_n|\nonumber\\
 &+&2\langle g(X_n)\Delta W_n, h(X_n)\Delta N_n\rangle.
 \label{ch5meantamed11}
 \end{eqnarray}
 
 Taking the expectation in both sides of \eqref{ch5meantamed11},  
 using the relation $\mathbb{E}\|\Delta W_n\|=0$, $\mathbb{E}\|\Delta W_n\|^2=\Delta t$, $\mathbb{E}|\Delta N_n|=\lambda\Delta t$ and $\mathbb{E}|\Delta N_n|^2=\lambda^2\Delta t^2+\lambda\Delta t$ 
 leads to 
 \begin{eqnarray*}
 \mathbb{E}\|X_{n+1}\|^2&\leq& \mathbb{E}\|X_n\|^2+K\Delta t \mathbb{E}\|X_n\|^2+2K^2\Delta t^2\mathbb{E}\|X_n\|^2+\theta^2\Delta t\mathbb{E}\|X_n\|^2+\lambda^2C^2\Delta t^2\mathbb{E}\|X_n\|^2\\
 &+&\lambda C^2\Delta t\mathbb{E}\|X_n\|^2
 -2\mu\lambda\Delta t\mathbb{E}\|X_n\|^2+2\lambda CK\Delta t\mathbb{E}\|X_n\|^2\\
 &=&\left[1+(2K^2+\lambda^2C^2)\Delta t^2+(K+\theta^2+\lambda C^2-2\mu\lambda+2\lambda CK)\Delta t\right]\mathbb{E}\|Y_n\|^2.
 \end{eqnarray*}
 Iterating the last inequality leads to 
 \begin{eqnarray*}
 \mathbb{E}\|X_n\|^2\leq \left[1+(2K^2+\lambda^2C^2)\Delta t^2+(K+\theta^2+\lambda C^2-2\mu\lambda+2\lambda CK)\Delta t\right]^n\mathbb{E}\|X_0\|^2.
 \end{eqnarray*}
 To have the stability of the NCTS scheme, we should also have
 \begin{eqnarray*}
 1+(2K^2+\lambda^2C^2)\Delta t^2+(K+\theta^2+\lambda C^2-2\mu\lambda+2\lambda CK)\Delta t<1.
 \end{eqnarray*}
 That is 
 \begin{eqnarray*}
 \Delta t<\dfrac{-[K+\theta^2+\lambda C^2-2\mu\lambda+2\lambda CK]}{2K^2+\lambda^2C^2},
 \end{eqnarray*}
 and there exists  a constant  $\gamma=\gamma (\Delta t) >0$ such that
 \begin{eqnarray*}
 \mathbb{E}\|X_n\|^2\leq \mathbb{E}\|X_0\|^2 e^{-\gamma \, t_n}, \,\, t_n= n\,\Delta t.
 \end{eqnarray*}
 As  in  the proof  of  \thmref{nt1}, we obviously have  $\underset{ \Delta t \rightarrow 0}{\lim}\gamma(\Delta t)=- (K +\theta^2+\lambda C(2 K+C)-2\mu \lambda)=-\alpha_3$.
 \end{proof}
 \begin{remark}
  Note that from the studies  above,   we can  deduce the  linear stabilities of schemes  
  \eqref{ncts} and \eqref{sts} for SDEs  without jump by setting $c=0$. 
  However, by setting  $h=0$  we  obtain the  nonlinear stability of semi-tamed scheme \eqref{sts} without jump performed in  \cite{lamp}.
  \end{remark}

 \section{Numerical simulations}
 \label{simulations}
 The goal of this section is to provide some practical examples to sustain our theoretical results in  the previous section.
 We compare  the stability behaviors of the tamed scheme and the semi-tamed scheme with  those of  numerical schemes presented in \cite{Desmond2}.
 More precisely, we test the stability of the semi-tamed scheme, tamed scheme, backward Euler and split-step backward Euler schemes with different stepsizes.
 We denote by $Y_n$ all the approximated solutions from  those schemes.
 \subsection{Linear case}
 Here we consider the following linear stochastic differential equation 
 \begin{eqnarray*}
  dX_t=-X_tdt+2X_t dW_t-0.9X_tdN_t, \hspace{0.5cm} t\in (0,1],\quad X_0=1,
 \end{eqnarray*}
 where  the intensity of the poisson process is taken  to be  $\lambda=9$. So 
 $l=2a+b^2+\lambda c(2+c)=-0.91<0$, which  ensures the  linear mean-square stable of  the exact solution.
 We can  easily check from  the  theoretical results  in  the previous section  that for $\Delta t<0.08$ the semi-tamed and the tamed Euler scheme reproduce 
 the linear  mean-square property of the exact solution. In \figref{FIG02}, we illustrate the mean-square stability of the semi-tamed scheme, 
 the tamed scheme, the Backward and the Spli-Step backward Euler scheme \cite{Desmond2} for different stepsizes. 
 We take $\Delta t=0.02, 0.01$ and $0.005$,  and generate $5\times 10^3$ paths for each numerical method. 
 We can observe from   \figref{FIG02} that the semi-tamed  scheme works well with the Backward and  Spli-Step backward Euler schemes.
 We can also observe that the semi-tamed  scheme works better than the tamed Euler scheme, and in some case overcomes the Backward and  Spli-Step backward Euler schemes.
 
\subsection{Nonlinear case}
For nonlinear stability,  we consider the following nonlinear stochastic differential equation
\begin{eqnarray*}
dX_t=\left(-\dfrac{12}{100}X_t-X_t^3 \right)dt+\dfrac{1}{10}X_tdW_t-\dfrac{1}{10}dN_t, \hspace{0.5cm} t\in(0,2],\quad \quad X_0=1.
\end{eqnarray*}
The poisson process  intensity is $\lambda=1$, $f(x)=-\dfrac{12}{100}x-x^3, \; g(x)=\dfrac{1}{10}x$ and $h(x)=-\dfrac{1}{10}x$.
We take  $u(x)=-\dfrac{12}{100}x$ and $v(x)=-x^3$. Indeed, we obviously have 
\begin{eqnarray*}
\langle x-y, f(x)-f(y)\rangle \leq -\dfrac{12}{100}(x-y)^2\\
|g(x)-g(y)|^2\leq \dfrac{1}{100}(x-y)^2, \,\;\;
|h(x)-h(y)|^2\leq \dfrac{1}{100}(x-y)^2.
\end{eqnarray*}
Then  $\mu=-\dfrac{12}{100}$, $\sigma=\gamma=\dfrac{1}{100}$ and  $\alpha=2\mu+\sigma+\lambda\sqrt{\gamma}(\sqrt{\gamma}+2)=-\dfrac{1}{50}<0$.
 It follows  that the exact solution is exponentially mean-square stable.
 One can easily check that for $\Delta t<0.41$ from  theoretical results the semi-tamed and the tamed Euler reproduce the  exponentially mean-square stability property of the exact solution. 
 \figref{FIG03} illustrates the stability of the tamed  scheme, the semi-tamed  scheme, 
 the Backward Euler and the Split-Step Backward Euler for different stepsizes. We take $\Delta t =0.04, 0.02$ and $0.01$ and  generate $3\times 10^3$ samples for each numerical method.
 We can observe that all  the schemes have the same stability behavior, although tamed and semi-tamed schemes are very efficient and  Backward Euler and the Split-Step Backward Euler
 less efficient as nonlinear equations are solved for each time step. 
 
\section*{Acknowledgements}
This project was supported by the Robert Bosch
Stiftung through the AIMS ARETE chair programme.

\begin{figure}
\begin{center}
  \subfigure[]{
\label{FIG02a}
   \includegraphics[width=0.48\textwidth]{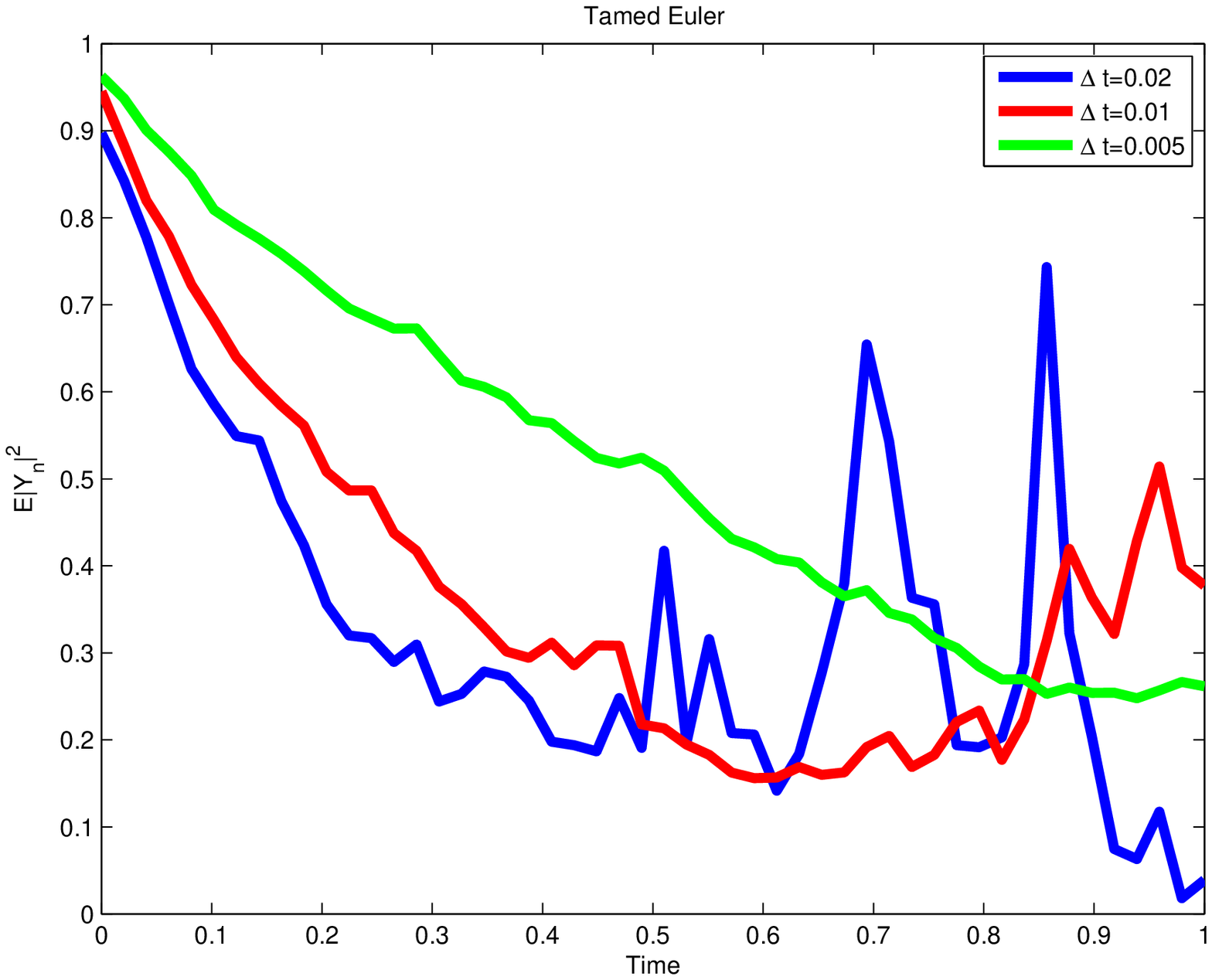}}
   \hskip 0.01\textwidth
   \subfigure[]{
   \label{FIG02b}
   \includegraphics[width=0.48\textwidth]{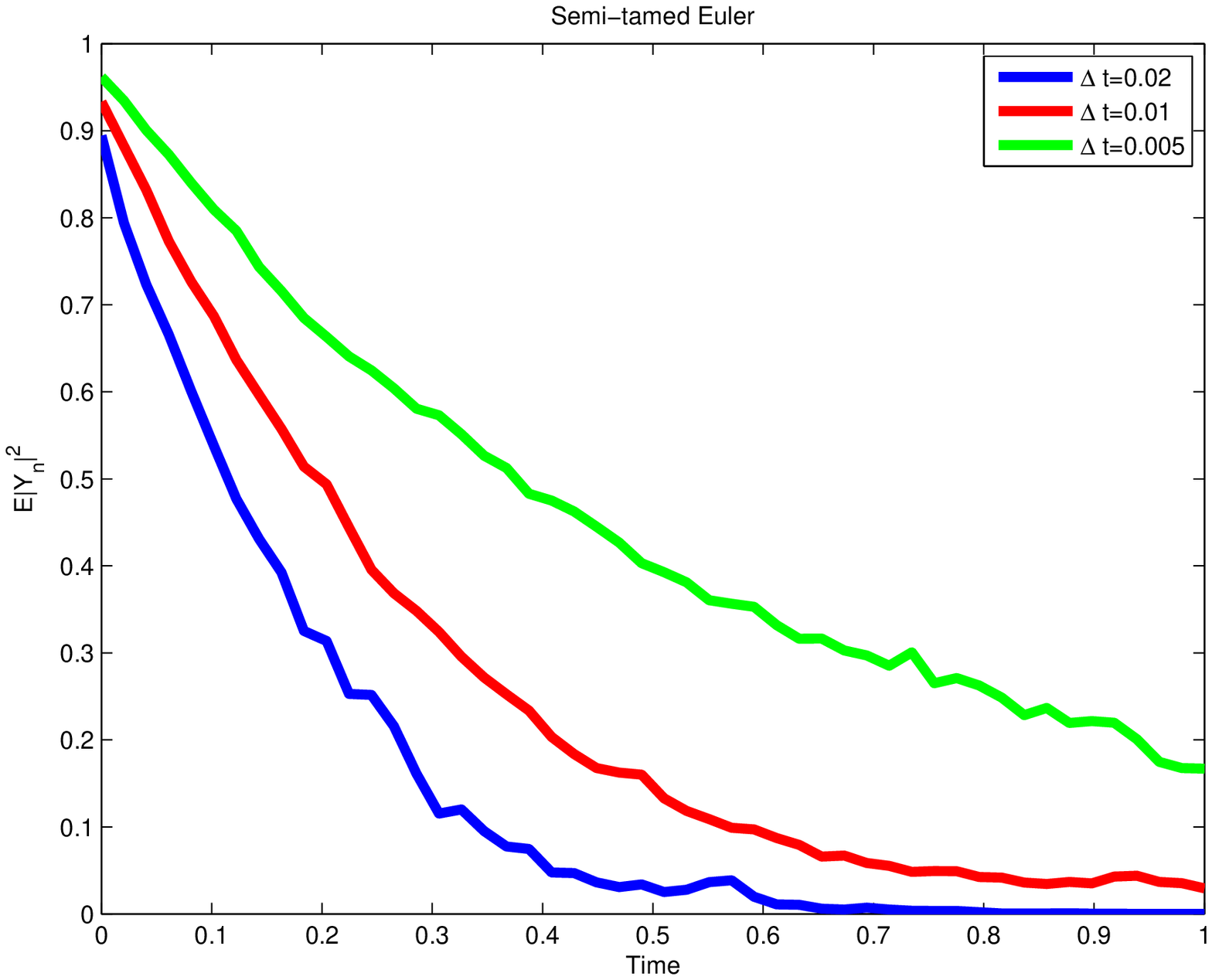}}
   \hskip 0.01\textwidth
\subfigure[]{
   \label{FIG02c}
   \includegraphics[width=0.48\textwidth]{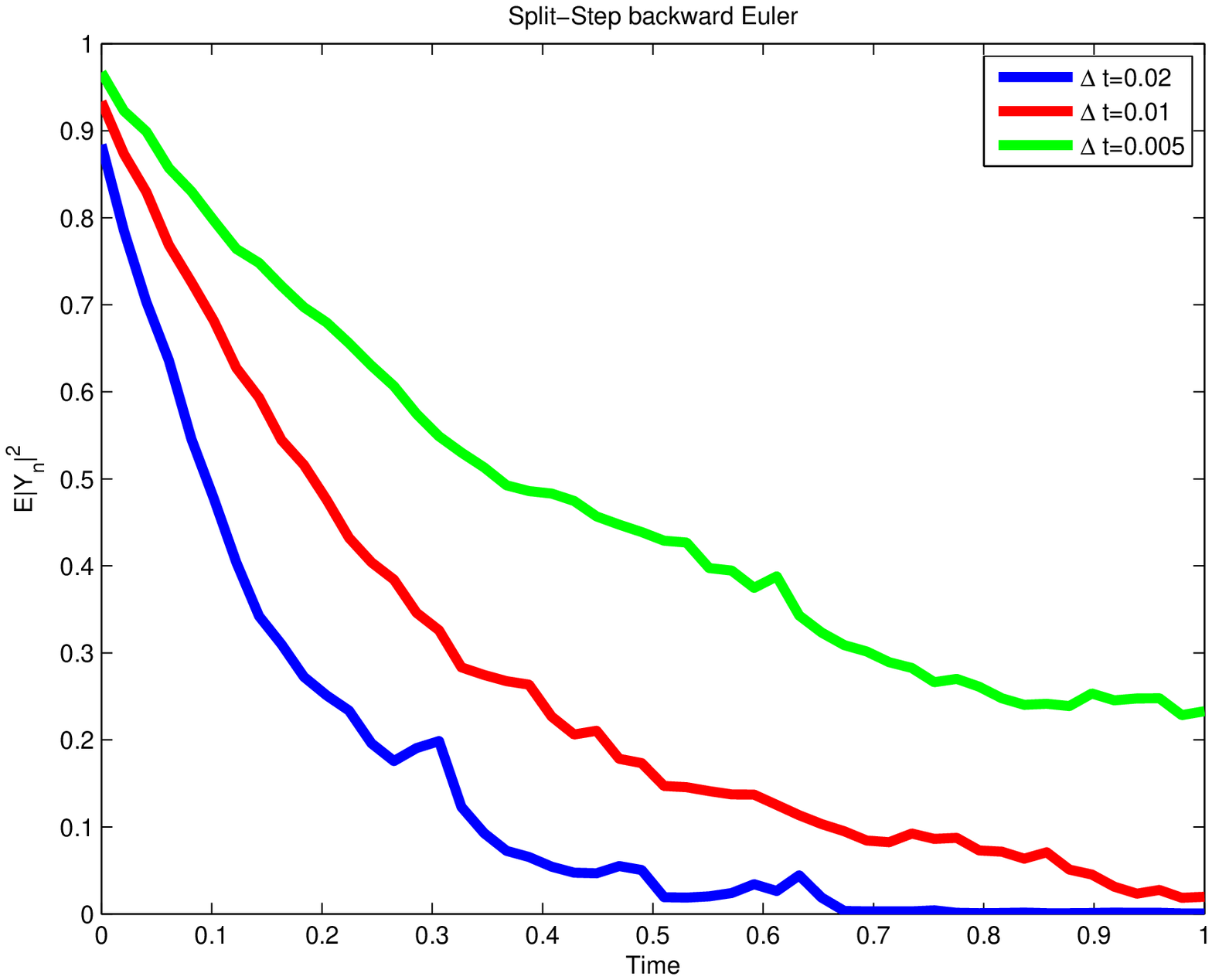}}
   \hskip 0.01\textwidth
\subfigure[]{
   \label{FIG02d}
   \includegraphics[width=0.48\textwidth]{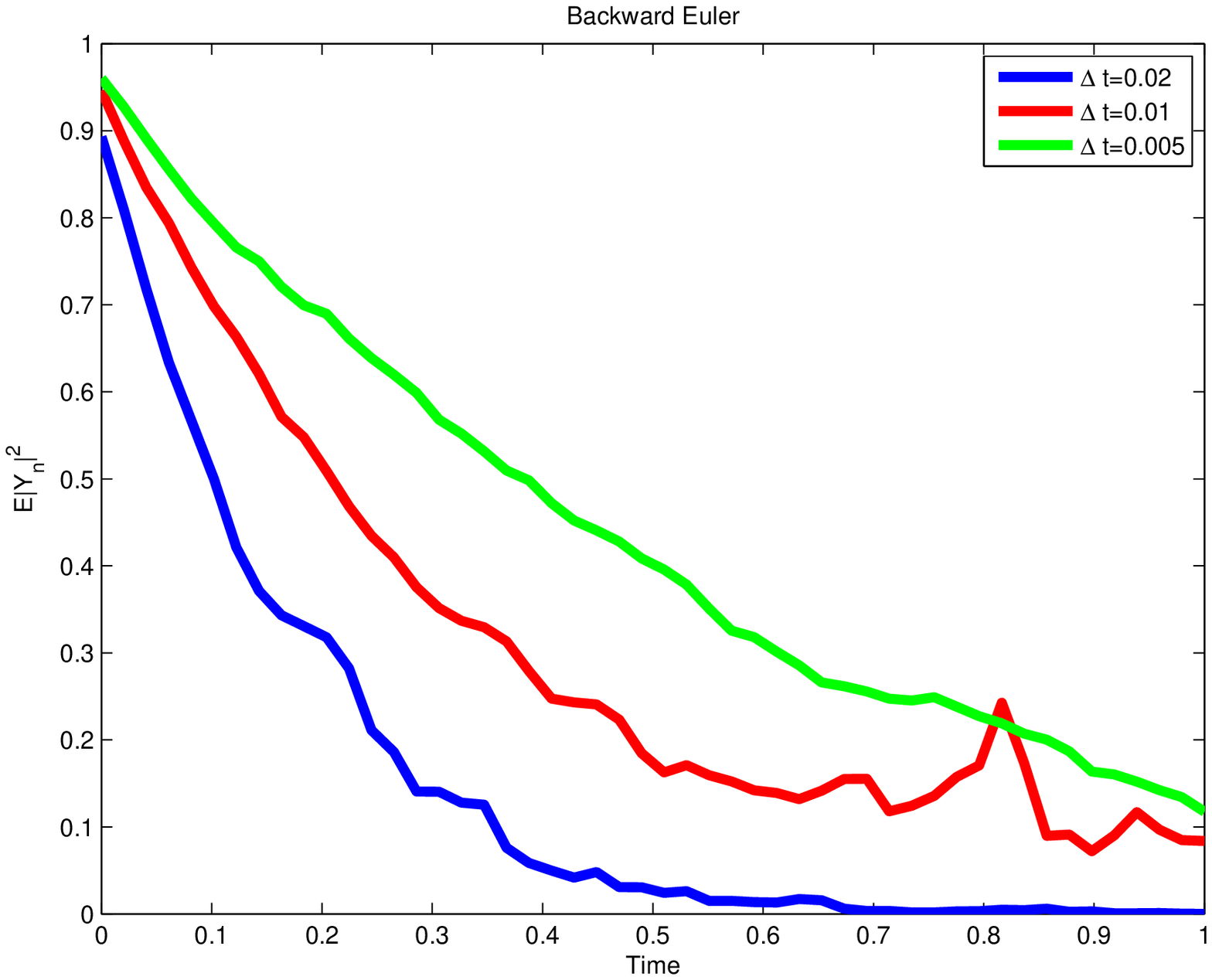}}
\caption{ Linear stability of the tamed scheme (a),  semi-tamed scheme (b),  Split-Step Backward Euler (c) and  Backward Euler (d) for
$\Delta t =0.02, 0.01, 0.005$ and  $5\times 10^3$ samples of each numerical scheme.}
 \label{FIG02}
 \end{center}
 \end{figure}
 \clearpage
\begin{figure}
 \begin{center}
  \subfigure[]{
\label{FIG03a}
   \includegraphics[width=0.48\textwidth]{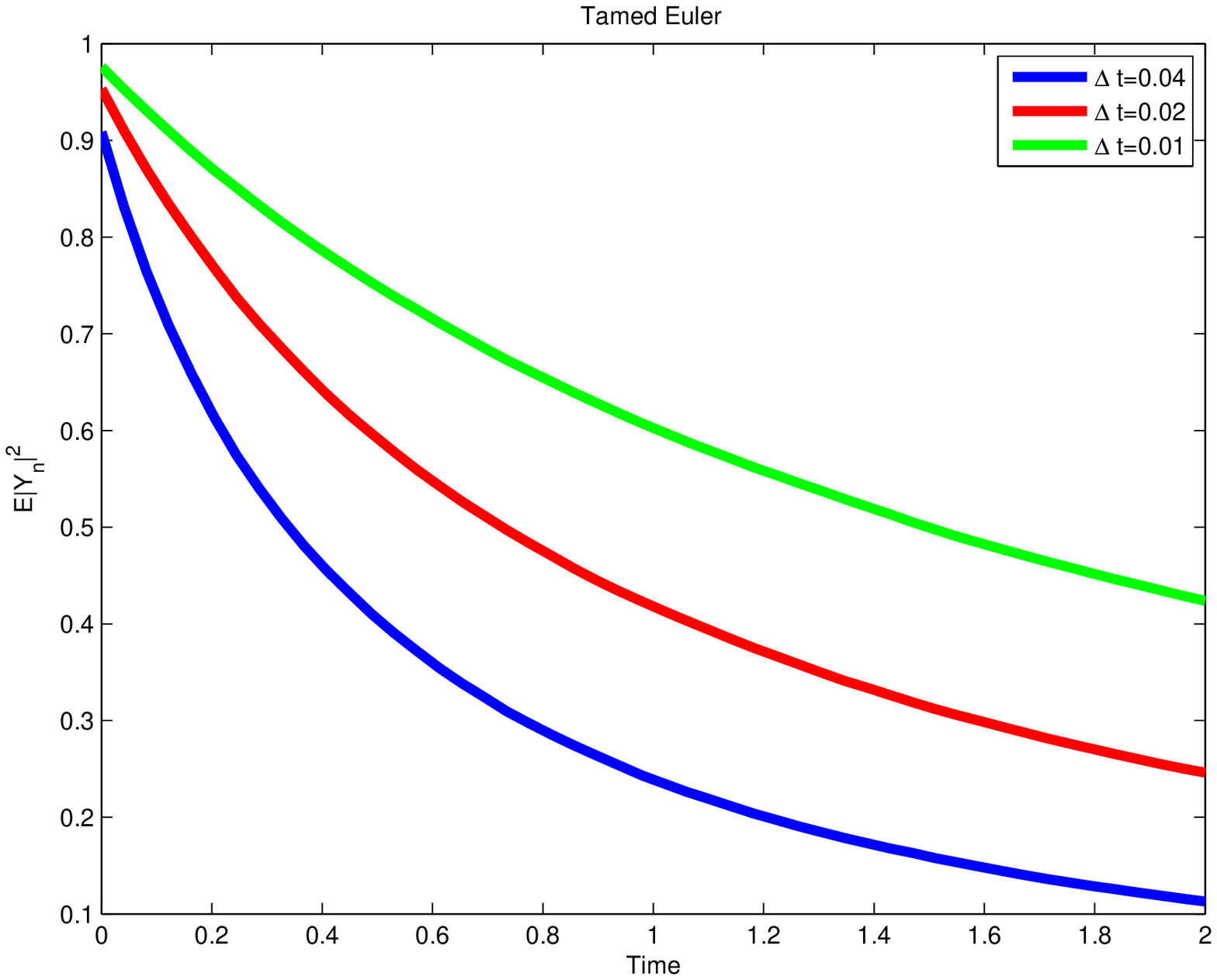}}
   \hskip 0.01\textwidth
   \subfigure[]{
   \label{FIG03b}
   \includegraphics[width=0.48\textwidth]{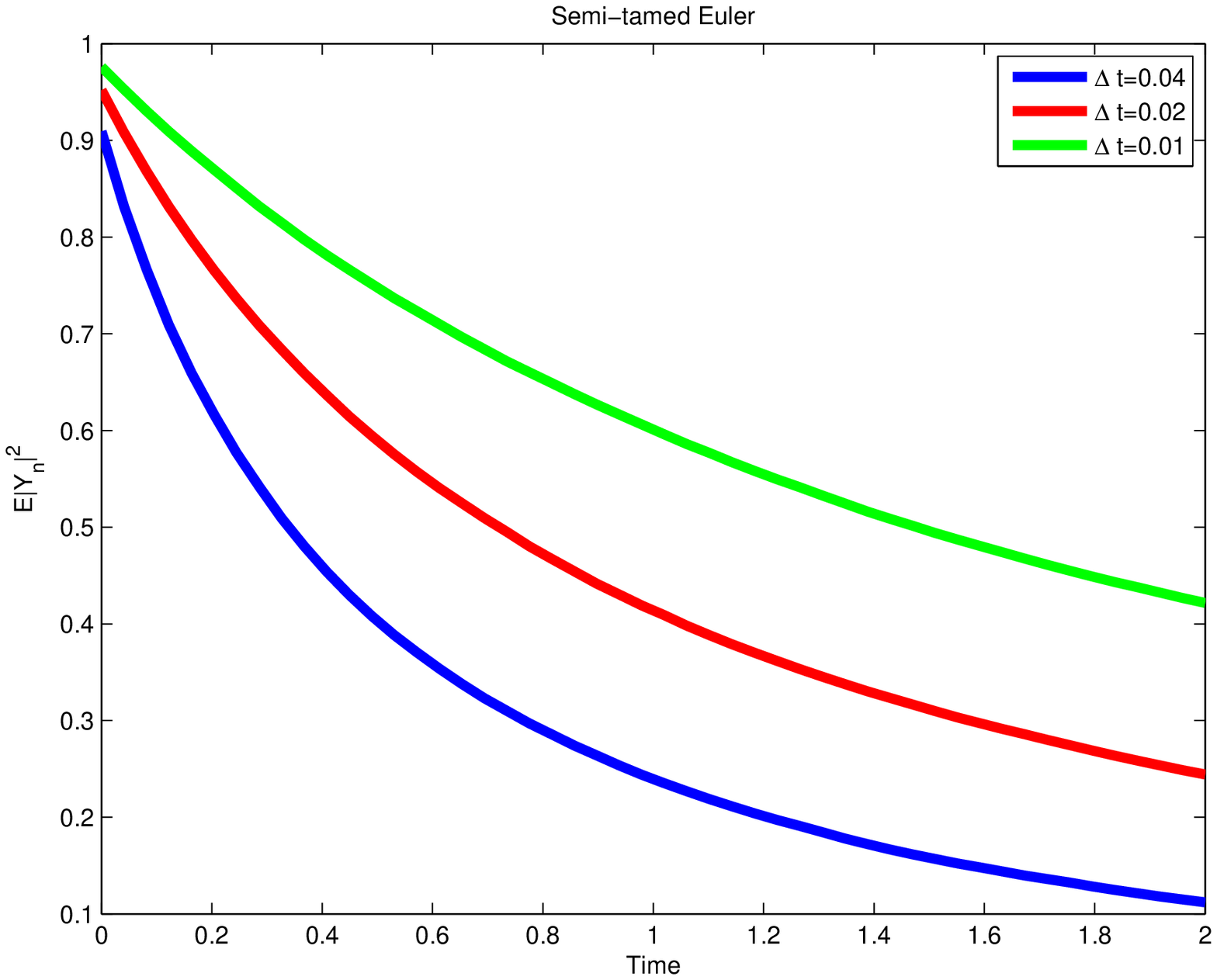}}
   \hskip 0.01\textwidth
\subfigure[]{
   \label{FIG03c}
   \includegraphics[width=0.48\textwidth]{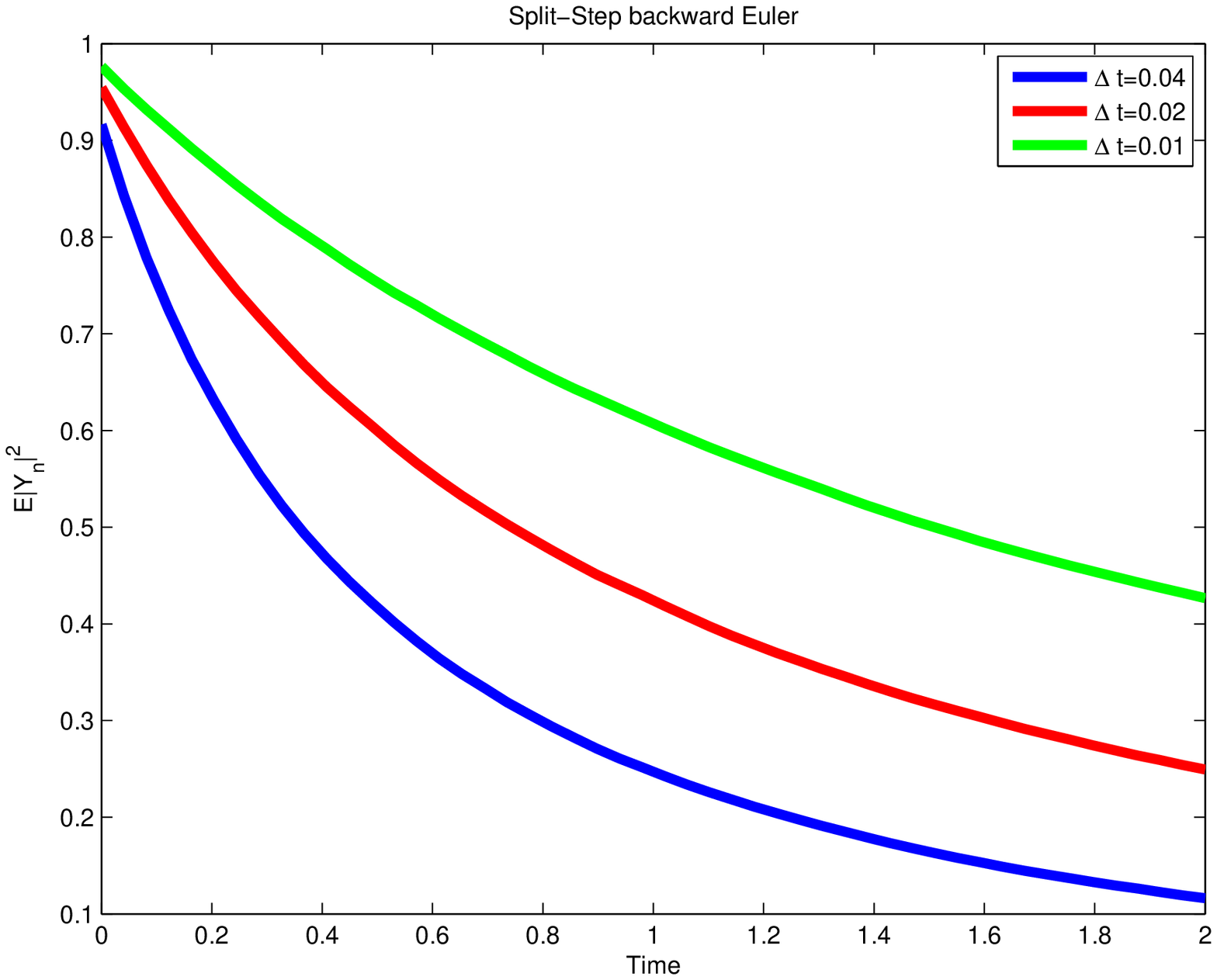}}
   \hskip 0.01\textwidth
\subfigure[]{
   \label{FIG03d}
   \includegraphics[width=0.48\textwidth]{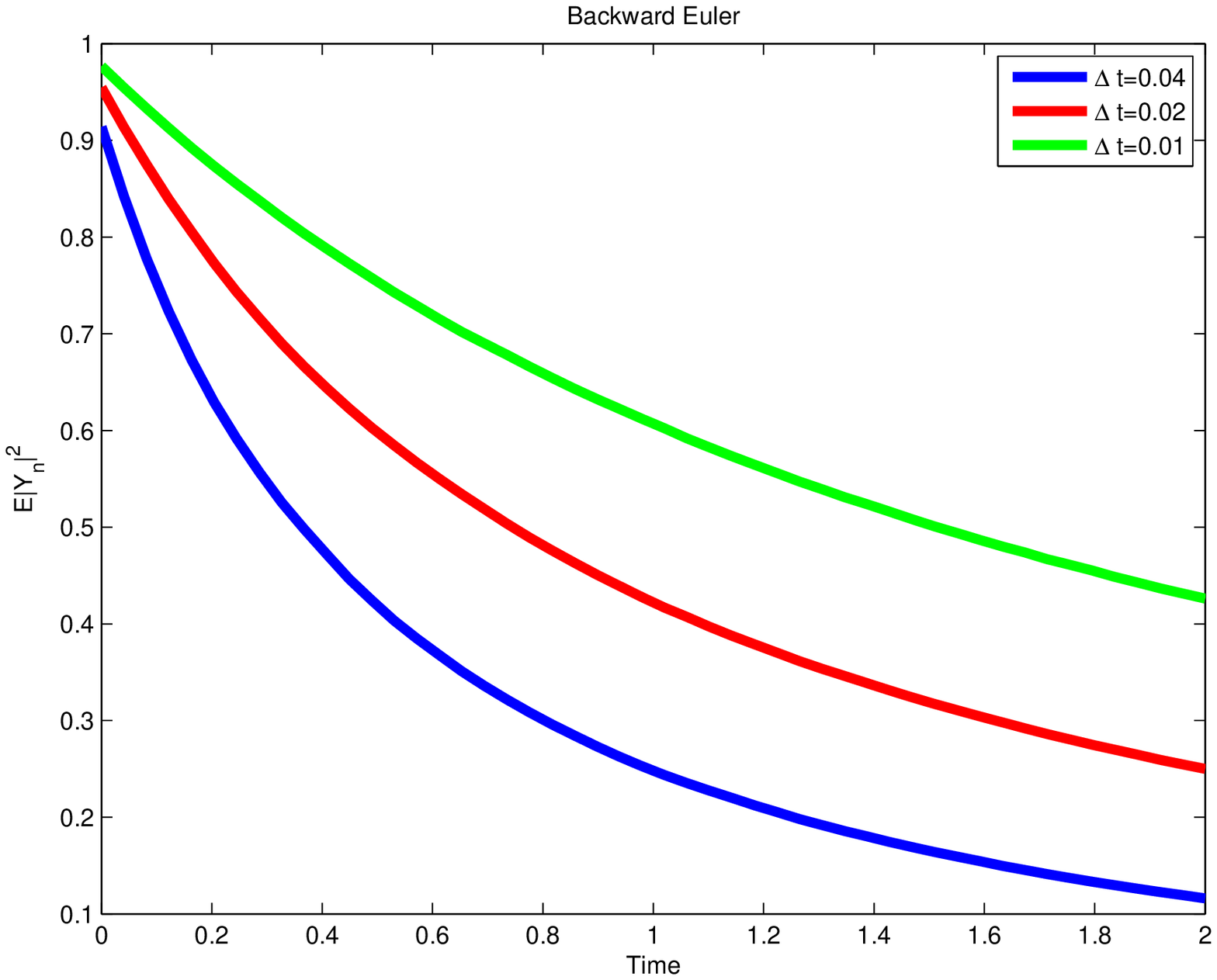}}
\caption{ Nonlinear stability of the tamed scheme (a),  semi-tamed scheme (b),  Split-Step Backward Euler (c) and  Backward Euler (d) for
$\Delta t =0.04, 0.01, 0.01$ and  $3\times 10^3$ samples of each numerical scheme.}
 \label{FIG03}
 \end{center}
\end{figure}
\clearpage
\end{document}